\documentclass{amsart}
\usepackage{amsmath,amscd,amssymb,amsthm,amsbsy,amscd,stmaryrd}
\usepackage[dvips]{graphicx,color}
\usepackage{epsfig}
\usepackage{mathrsfs}
\usepackage{mathpazo}
\input{xy}

\newtheorem{theorem}{Theorem}[section]
\newtheorem{lemma}[theorem]{Lemma}
\newtheorem{proposition}[theorem]{Proposition}

\theoremstyle{definition}

\numberwithin{equation}{section}

\begin{document}
\title{Curvature estimates on a parabolic flow of Fei-Guo-Phong}
%%%%%%%%%%%%%%%%%%%%%%%%%%%%%%%%%%%%%%%%%%%%%%%%%%%%%%%%%%%%%%%%%%%%%%%%%%%%%%%%

\author{Yi Li}
\address{Shing-Tung Yau Center of Southeast University, Sipailou Campus,
Xuanwu District, Nanjing, 210000, China; and School of Mathematics,
Southeast University, Li Wenzheng Library of Jiulonglu Campus, Jiangning District, Nanjing, 211189, China}
\email{yilicms@gmail.com}

%\author{Yuan Yuan}
%\address{Department of Mathematics, Syracuse University, Syracuse, NY %13244, USA}
%\email{ yyuan05@syr.edu}

%\subjclass[2010]{Primary 53C44, 53C10}
%\keywords{Generalized Ricci flow, generalized W-entropy, Li-Yau-Hamilton Harnack estimates}

\maketitle

%%%%%%%%%%%%%%%%%%%%%%%%%%%%%%%%%%%%%%%%%%%%%%%%%%%%%%%%%%%%%%%%%%%%%%%%%%%%%%%%
\begin{abstract} In \cite{FGP}, Fei, Guo and Phong established a criteria for the long-time existence of their parabolic flow from $11$-dimensional supergravity, which involves Riemannian curvatures ${\rm Rm}(g(t))$ and 4-forms $F(t)$. In this paper, we obtain a new criteria for the long-time existence of the same flow, which involves only Ricci curvatures ${\rm Ric}(g(t))$, $F(t)$, but as well as $\nabla_{g(t)}F(t)$.
\end{abstract}

\renewcommand{\labelenumi}{Case \theenumi.} %
\newtheoremstyle{mystyle}{3pt}{3pt}{\itshape}{}{\bfseries}{}{5mm}{} %
\theoremstyle{mystyle} \newtheorem{Thm}{Theorem} \theoremstyle{mystyle} %
\newtheorem{lem}{Lemma} \newtheoremstyle{citing}{3pt}{3pt}{\itshape}{}{%
\bfseries}{}{5mm}{\thmnote{#3}} \theoremstyle{citing} %
\newtheorem*{citedthm}{}

\tableofcontents

%%%%%%%%%%%%%%%%%%%%%%%%%%%%%%%%%%%%%%%%%%%%%%%%%%%%%%%%%%%%%%%%%%%%%%%%%%%%%%
%%%%%%%%%%%%%%%%%%%%%%%%%%%%%%%%%%%%%%%%%%%%%%%%%%%%%%%%%%%%%%%%%%%%%%%%%%%%%%
\section{Introduction}\label{section1}
%%%%%%%%%%%%%%%%%%%%%%%%%%%%%%%%%%%%%%%%%%%%%%%%%%%%%%%%%%%%%%%%%%%%%%%%%%%%%%
%%%%%%%%%%%%%%%%%%%%%%%%%%%%%%%%%%%%%%%%%%%%%%%%%%%%%%%%%%%%%%%%%%%%%%%%%%%%%%

In this paper, we consider a parabolic flow coming from 11D supergravity. Let $(M,g)$ be a $11$-dimensional closed Riemannian manifold with a closed $4$-form $F$. This parabiolic flow is given by
\begin{eqnarray}
\partial_{t}g(t)&=&-2\!\ {\rm Ric}_{g(t)}+F(t)\circ_{g(t)}F(t)-\frac{1}{3}|F(t)|^{2}_{g(t)}g(t),\label{1.1}\\
\partial_{t}F(t)&=&\Delta_{g(t), H}F(t)+\frac{1}{2}d\ast_{g(t)}(F(t)\wedge F(t)),\label{1.2}
\end{eqnarray}
where $\Delta_{g(t), H}=-(dd^{\ast}_{g(t)}+d^{\ast}_{g(t)}d)$ denotes the Hodge-Laplace operator associated to $g(t)$, $g(t)$ are smooth Riemannian metrics, and $F(t)$ are smooth closed $4$-forms. As in \cite{FGP}, 
\begin{eqnarray*}
(F(t)\circ_{g(t)}F(t))_{AB}&:=&\frac{1}{3!}g^{CC'}g^{DD'}g^{EE'}F_{ACDE}F_{BC'D'E}, \\
|F(t)|^{2}_{g(t)}&=&\frac{1}{4!}g^{AA'}g^{CC'}g^{DD'}g^{EE'}F_{ACDE}F_{A'C'D'E'},
\end{eqnarray*}
if $F(t)$ is locally written as $F(t)=\frac{1}{4!}F_{ABCD}dx^{A}\wedge dx^{B}\wedge dx^{C}
\wedge dx^{D}$.

\begin{theorem}\label{t1.1}{\bf (Fei-Guo-Phong \cite{FGP})} The flow (\ref{1.1}) -- (\ref{1.2}) exists in $[0,T_{\max})$, where $T_{\max}\in(0,+\infty]$ is the maximal time. 
\begin{itemize}
    
    \item[(a)] $F(t)$ remains closed for any $t\in[0,T_{\max})$. 
    
    \item[(b)] If $T_{\max}<\infty$, then
    \begin{equation}
    \limsup_{t\to T_{\max}}\sup_{M}\left(|{\rm Rm}_{g(t)}|_{g(t)}+|F(t)|_{g(t)}\right)
    =+\infty.\label{1.3}
    \end{equation}
\end{itemize}
\end{theorem}

In this paper, we improve the result (\ref{1.3}) in the following sense.

\begin{theorem}\label{t1.2} Consider the flow (\ref{1.1}) -- (\ref{1.2}) in $M\times[0,T_{\max})$, where $T_{\max}$ is as in Theorem \ref{t1.1}. If $T_{\max}<+\infty$, then
\begin{equation}
     \limsup_{t\to T_{\max}}\sup_{M}\left(|{\rm Ric}_{g(t)}|_{g(t)}+|F(t)|_{g(t)}+|\nabla_{g(t)}F(t)|_{g(t)}\right)
    =+\infty.\label{1.4}
\end{equation}
\end{theorem}

${}$

{\bf Acknowledgments.} The author wishes to thank Yuan Yuan for helpful discussion on this paper. The author also thank Teng Fei, Bin Guo, and D. H. Phong for their useful suggestions and comments.

%%%%%%%%%%%%%%%%%%%%%%%%%%%%%%%%%%%%%%%%%%%%%%%%%%%%%%%%%%%%%%%%%%%%%%%%%%%%%%
\subsection{Notions and conventions}\label{subsection1.1}
%%%%%%%%%%%%%%%%%%%%%%%%%%%%%%%%%%%%%%%%%%%%%%%%%%%%%%%%%%%%%%%%%%%%%%%%%%%%%%

If $g$ is a Riemannian metric on the manifold $M$, we write
${\rm Rm}_{g}, {\rm Ric}_{g}, R_{g}, \nabla_{g}, \nabla^{2}_{g},
\Delta_{g}, dV_{g}, {\rm div}_{g}, {\rm tr}_{g(t)}$, $\langle\cdot,\cdot\rangle_{g}$, and $|\cdot|_{g}$ the Riemann curvature, Ricci
curvature, scalar curvature, Levi-Civita connection, Hessian,
Beltrami-Laplace operator, volume form, divergence, trace, inner product, and norm of $g$,
respectively. In a local coordinate system, we write
\begin{equation*}
\Gamma^{k}_{ij}=\frac{1}{2}g^{k\ell}
\left(\frac{\partial}{\partial x^{i}}g_{\ell j}+\frac{\partial}{\partial x^{j}}
g_{i\ell}-\frac{\partial}{\partial x^{\ell}}g_{ij}\right).
\end{equation*}
The Riemann curvature ${\rm Rm}_{g}(X,Y)Z=[(\nabla_{g})_{X}, (\nabla_{g})_{Y}]Z-(\nabla_{g})_{[X,Y]}Z$ is locally given by
\begin{equation*}
R_{ijk}{}^{\ell}\frac{\partial}{\partial x^{\ell}}\equiv R^{\ell}_{ijk}\frac{\partial}{\partial x^{\ell}}:=
{\rm Rm}_{g}\left(\frac{\partial}{\partial x^{i}},\frac{\partial}{\partial x^{j}}
\right)\frac{\partial}{\partial x^{k}}.
\end{equation*}
Contracting indices we get $R_{ij}:=g^{k\ell}R_{ik\ell j}$ and $R:=
g^{ij}R_{ij}$.

We use the Einstein summation and raise/lower indices for tensor fields. The
Ricci identity for tensor $T=(T_{k_{1}\cdots k_{r}}{}^{\ell_{1}\cdots
\ell_{s}})$ is
$$
[\nabla_{i},\nabla_{j}]T_{k_{1}\cdots k_{r}}{}^{\ell_{1}
\cdots\ell_{s}} \ \ = \ \ -\sum_{1\leq h\leq r}R_{ijk_{h}}{}^{p}
T_{k_{1}\cdots k_{h-1}p k_{h+1}\cdots k_{r}}{}^{\ell_{1}\cdots\ell_{s}}
$$
$$
+ \ \sum_{1\leq h\leq s}R_{ijp}{}^{\ell_{h}}T_{k_{1}\cdots k_{r}}{}^{\ell_{1}
\cdots\ell_{h-1}p\ell_{h+1}\cdots\ell_{s}}
$$
$$
=-\sum_{1\leq h\leq r}R_{ijk_{h}p}T_{k_{1}\cdots k_{h-1}}{}^{p}{}_{k_{h+1}
\cdots k_{r}}{}^{\ell_{1}\cdots\ell_{s}}+\sum_{1\leq h\leq s}R_{ij}{}^{p}{}_{\ell_{h}}T_{k_{1}\cdots k_{r}}{}^{\ell_{1}
\cdots\ell_{h-1}}{}_{p}{}^{\ell_{h+1}\cdots\ell_{s}}.
$$
For any two tensor fields $A$ and $B$, we denote by $A\ast B$ any quantity
obtained from the tensor product $A\otimes B$ by one or more of these
operations:

\begin{itemize}

\item[(a)] summation over pairs of matching upper and lower indices,

\item[(b)] multiplication by constants depending only on the dimension
of $\mathcal{M}$ and the ranks of $A$ and $B$.

\end{itemize}
We also denote by $A^{\ast k}$ the $k$-fold product $A\ast \cdots\ast A$.

If $g=g(t)$ is a family of Riemannian metric indexed time variable $t$,
we always omit time $t$ in all components of curvatures, operators, and
other quantities induced from $g(t)$. The time derivative is denoted by $
\partial_{t}$ or $\partial/\partial t$. For the parabolic operator $
\partial_{t}-\Delta_{g(t)}$, we use the symbol $\Box_{g(t)}$.

For a given parabolic PDE, a uniform constant $C$ is a {\it positive} constant depending only on the dimension of $M$ as well as the given data, but not on time $t$. Different
uniform constants may be labeled by $C_{1}, C_{2},\cdots$, according to
the context.

If $P$ and $Q$ are two quantities (may depend on time $t$) satisfying $P
\leq C Q$ for some uniform constant $C$, then we set $P\lesssim Q$. Similarly, we can define $P\approx Q$ if $P\lesssim Q$ and
$Q\lesssim P$; that is, $C^{-1}Q\leq P\leq CQ$ for some uniform
constant $C$.

%%%%%%%%%%%%%%%%%%%%%%%%%%%%%%%%%%%%%%%%%%%%%%%%%%%%%%%%%%%%%%%%%%%%%%%%%%%%%%
%%%%%%%%%%%%%%%%%%%%%%%%%%%%%%%%%%%%%%%%%%%%%%%%%%%%%%%%%%%%%%%%%%%%%%%%%%%%%%
\section{Evolution equations}\label{section2}
%%%%%%%%%%%%%%%%%%%%%%%%%%%%%%%%%%%%%%%%%%%%%%%%%%%%%%%%%%%%%%%%%%%%%%%%%%%%%%
%%%%%%%%%%%%%%%%%%%%%%%%%%%%%%%%%%%%%%%%%%%%%%%%%%%%%%%%%%%%%%%%%%%%%%%%%%%%%%

In this section, for further study, we assume that $(g(t), F(t))_{t\in[0,T]}$ is a solution to (\ref{1.1}) -- (\ref{1.2}) on a {\it complete} Riemannian manifold $(M, g)$ 
with a closed $4$-form $F$.

${}$

For convenience, we introduce 
\begin{equation}
    \alpha(t):=F(t)\circ_{g(t)}F(t)-\frac{1}{3}|F(t)|^{2}_{g(t)}g(t).\label{2.1}
\end{equation}
Then the equation (\ref{1.1}) becomes
\begin{equation}
    \partial_{t}g(t)=-2\!\ {\rm Ric}_{g(t)}+\alpha(t).\label{2.2}
\end{equation}
It is not hard to see that our $\alpha(t)$ does not satisfy the equation (2.2) in \cite{LY} (because the evolution equation for $\Box_{g(t)}\alpha(t)$ will contain a term involving $\nabla_{g(t)}F(t)$), so that the main result in \cite{LY} can not be applied to (\ref{1.1}) -- (\ref{1.2}). However many general equations in Section 2 of \cite{LY} can be used in our setting. 

${}$

From now on, we assume
\begin{equation}
    |{\rm Ric}_{g(t)}|_{g(t)}\leq K, \ \ \ |F(t)|_{g(t)}\leq L, \ \ \ \text{in} \ 
    M\times[0,T].\label{2.3}
\end{equation}
Firstly, we rewrite (\ref{1.1}) and (\ref{1.2}) as
\begin{eqnarray}
\partial_{t}g(t)&=&-2\!\ {\rm Ric}_{g(t)}+F(t)\ast F(t),\label{2.4}\\
\partial_{t}F(t)&=&\Delta_{g(t)}F(t)+\nabla_{g(t)} F(t)\ast F(t)+
F(t)\ast F(t).\label{2.5}
\end{eqnarray}
We can deduce from \cite{FGP, LY} the following evolution equations or inequalities:
\begin{eqnarray}
|\nabla{\rm Ric}|^{2}&\leq&-\frac{1}{2}\Box|{\rm Ric}|^{2}+C|{\rm Ric}|^{2}|{\rm Rm}|
+C|{\rm Ric}|^{2}|F|^{2}\nonumber\\
&&+ \ {\rm Ric}\ast\left(\nabla^{2}F\ast F+\nabla F\ast\nabla F\right)\nonumber\\
&\leq&-\frac{1}{2}\Box|{\rm Ric}|^{2}+CK^{2}|{\rm Rm}|
+CK^{2}L^{2}+ {\rm Ric}\ast\left(\nabla^{2}F\ast F+\nabla F\ast\nabla F\right),\label{2.6}
\end{eqnarray}
\begin{eqnarray}
|\nabla{\rm Rm}|^{2}&\leq&-\frac{1}{2}\Box|{\rm Rm}|^{2}
+C|{\rm Rm}|^{3}+{\rm Rm}\ast{\rm Rm}\ast F\ast F\nonumber\\
&&+ \ {\rm Rm}\ast\left(\nabla^{2}F\ast F+\nabla F\ast\nabla F\right)\nonumber\\
&\leq&-\frac{1}{2}\Box|{\rm Rm}|^{2}+C|{\rm Rm}|^{3}+CL^{2}|{\rm Rm}|^{2}
+{\rm Rm}\ast\left(\nabla^{2}F\ast F+\nabla F\ast\nabla F\right),\label{2.7}
\end{eqnarray}
\begin{eqnarray}
\partial_{t}|{\rm Rm}|^{2}&=&\nabla^{2}{\rm Ric}\ast{\rm Rm}
+{\rm Ric}\ast{\rm Rm}\ast{\rm Rm}\nonumber\\
&&+ \ \left({\rm Rm}\ast{\rm Rm}\ast F\ast F+{\rm Rm}\ast\nabla^{2}F\ast F
+{\rm Rm}\ast\nabla F\ast\nabla F\right),\label{2.8}
\end{eqnarray}
and
\begin{eqnarray}
|\nabla F|^{2}&\leq&-\frac{1}{2}\Box|F|^{2}+C|{\rm Rm}||F|^{2}+C|\nabla F||F|^{2},\label{2.9}\\
\partial_{t}dV_{t}&\leq&C(K+L^{2})dV_{t}.\label{2.10}
\end{eqnarray}
According to the Cauchy-Schwarz inequality, we can simplify (\ref{2.9}) into
\begin{equation}
    |\nabla F|^{2}\leq-\Box|F|^{2}+CL^{2}|{\rm Rm}|+CL^{4}.\label{2.11}
\end{equation}

\begin{proposition}\label{p2.1} Under the assumption (\ref{2.3}) we have
\begin{equation}
    \Box u\leq Cu^{2}\label{2.12}
    \end{equation}
    for some positive constants $C$, where $u:=1+|{\rm Rm}|^{2}+|F|^{2}+|\nabla F|^{2}$.
\end{proposition}

\begin{proof} From (\ref{2.7}) we have
\begin{eqnarray*}
\partial_{t}|{\rm Rm}|^{2}&\leq&\Delta|{\rm Rm}|^{2}-2|\nabla{\rm Rm}|^{2}
+C|{\rm Rm}|^{3}+C|F|^{2}|{\rm Rm}|^{2}\\
&&+ \ {\rm Rm}\ast(\nabla^{2}F\ast F+\nabla F\ast\nabla F)\\
&\leq&\Delta|{\rm Rm}|^{2}-2|\nabla{\rm Rm}|^{2}+C|{\rm Rm}|^{3}\\
&&+ \ C|F|^{2}|{\rm Rm}|^{2}
+C|{\rm Rm}||F||\nabla^{2}F|+C|{\rm Rm}||\nabla F|^{2}
\end{eqnarray*}
so that
\begin{eqnarray*}
\partial_{t}(|{\rm Rm}|^{2}+|F|^{2})&\leq&\Delta|{\rm Rm}|^{2}-2|\nabla{\rm Rm}|^{2}+C|{\rm Rm}|^{3}\\
&&+ \ C|F|^{2}|{\rm Rm}|^{2}
+C|{\rm Rm}||F||\nabla^{2}F|+C|{\rm Rm}||\nabla F|^{2}\\
&&+ \ \bigg(\Delta|F|^{2}-|\nabla F|^{2}+C|F|^{2}|{\rm Rm}|+C|F|^{4}\bigg)\\
&\leq&\Delta(|{\rm Rm}|^{2}+|F|^{2})-2|\nabla{\rm Rm}|^{2}-|\nabla F|^{2}\\
&&+ \ C|F||{\rm Rm}||\nabla^{2}F|+C|{\rm Rm}||\nabla F|^{2}\\
&&+ \ C|{\rm Rm}|^{3}+C|F|^{2}|{\rm Rm}|^{2}
+C|F|^{2}|{\rm Rm}|+C|F|^{4}.
\end{eqnarray*}
In order to eliminate $|\nabla^{2}F|$, we need the evolution equation to $|\nabla F|^{2}$. Compute
\begin{eqnarray*}
\partial_{t}|\nabla F|^{2}&=&2\langle\partial_{t}\nabla F, \nabla F\rangle+\partial_{t}g\ast\nabla F\ast\nabla F\\
&=&2\langle\nabla\partial_{t}F+\partial_{t}\Gamma\ast F,\nabla F\rangle+({\rm Ric}+F\ast F)\ast\nabla F\ast\nabla F\\
&=&2\langle\nabla(\Delta F+\nabla F\ast F+F\ast F)+\nabla({\rm Ric}+F\ast F)\ast F,\nabla F\rangle\\
&&+ \ ({\rm Ric}+F\ast F)\ast\nabla F\ast\nabla F\\
&=&2\langle\nabla\Delta F,\nabla F\rangle+F\ast\nabla F\ast\nabla^{2}F+\nabla F\ast\nabla F\ast\nabla F+F\ast\nabla F\ast\nabla F\\
&&+ \ \nabla{\rm Ric}\ast\nabla F+{\rm Ric}\ast\nabla F\ast\nabla F
+F\ast F\ast\nabla F\ast\nabla F.
\end{eqnarray*}
Using
$$
\Delta\nabla F=\nabla\Delta F+{\rm Rm}\ast\nabla F+\nabla{\rm Rm}\ast F
$$
and
$$
\Delta|\nabla F|^{2}=2\langle\Delta\nabla F,\nabla F\rangle+2|\nabla|\nabla F||^{2}\leq2\langle\Delta\nabla F, \nabla F\rangle+2|\nabla^{2}F|^{2}
$$
we obtain
\begin{eqnarray*}
\partial_{t}|\nabla F|^{2}&\leq&\Delta|\nabla F|^{2}-2|\nabla^{2}F|^{2}
+C|{\rm Rm}||\nabla F|^{2}+C|\nabla{\rm Rm}||F||\nabla F|\\
&&+ \ C|F||\nabla F||\nabla^{2}F|+C|\nabla F|^{3}
+C|F||\nabla F|^{2}\\
&&+ \ C|\nabla{\rm Ric}||\nabla F|+C|{\rm Ric}||\nabla F|^{2}+C|F|^{2}|\nabla F|^{2}.
\end{eqnarray*}
Therefore
\begin{eqnarray*}
\Box(|{\rm Rm}|^{2}+|F|^{2}+|\nabla F|^{2})
&\leq&-2|\nabla{\rm Rm}|^{2}+C|\nabla{\rm Rm}||F||\nabla F|+C|\nabla{\rm Ric}||\nabla F|\\
&&- \ 2|\nabla^{2}F|^{2}+C|F||{\rm Rm}||\nabla^{2}F|+C|F||\nabla F||\nabla^{2}F|\\
&&+ \ C|{\rm Rm}||\nabla F|^{2}+C|{\rm Rm}|^{3}+C|F|^{2}|{\rm Rm}|^{2}\\
&&+ \ C|F|^{2}|{\rm Rm}|+C|F|^{4}+C|\nabla F|^{3}+C|F||\nabla F|^{2}\\
&&+ \ C|{\rm Ric}||\nabla F|^{2}+C|F|^{2}|\nabla F|^{2}.
\end{eqnarray*}
By using the inequality $a\leq a^{2}+1$ for any $a\geq0$, we obtain (\ref{2.12}).
\end{proof}

Now we consider the quantity
\begin{equation}
    \frac{d}{dt}\left(\int|{\rm Rm}|^{p}\phi^{2p}dV_{t}\right), \ \ \ \int:=\int_{M}, \ \ \ p\geq5,\label{2.13}
\end{equation}
where $\phi = \phi(x)$ is a cutoff function with compact support $\Omega$ inside $M$. From (\ref{2.8}), (\ref{2.10}) and the identity $\partial_{t}|Rm|^{p}=\frac{p}{2}|{\rm Rm}|^{p-2}\partial_{t}|{\rm Rm}|^{2}$, we get
$$
\frac{d}{dt}\left(\int|{\rm Rm}|^{p}\phi^{2p}dV_{t}\right) \ \ = \ \ \int\left(\partial_{t}|{\rm Rm}|^{p}\right)\phi^{2p}dV_{t}+\int|{\rm Rm}|^{p}\phi^{2p}(\partial_{t}dV_{t})
$$
$$
\leq \ C(K+L^{2})\int|{\rm Rm}|^{p}\phi^{2p}dV_{t}+\frac{p}{2}\int |{\rm Rm}|^{p-2}
\phi^{2p}\bigg[\nabla^{2}{\rm Ric}\ast{\rm Rm}+{\rm Ric}\ast{\rm Rm}\ast{\rm Rm}
$$
$$
+ \ ({\rm Rm}\ast{\rm Rm}\ast F\ast F+{\rm Rm}\ast\nabla^{2}F\ast F+{\rm Rm}
\ast\nabla F\ast\nabla F)\bigg]dV_{t}
$$
$$
\leq \ C(K+L^{2})\int|{\rm Rm}|^{p}\phi^{2p}dV_{t}+C\int(\nabla F\ast\nabla F\ast{\rm Rm})|{\rm Rm}|^{p-2}\phi^{2p}dV_{t}
$$
$$
+ \ C\int(\nabla^{2}F\ast{\rm Rm}\ast F)|{\rm Rm}|^{p-2}\phi^{2p}dV_{t}+C\int(\nabla^{2}{\rm Ric}\ast{\rm Rm})|{\rm Rm}|^{p-2}\phi^{2p}dV_{t}.
$$
As in \cite{LY}, we compute, by Kato's inequality,
$$
C\int(\nabla^{2}{\rm Ric}\ast{\rm Rm})|{\rm Rm}|^{p-2}\phi^{2p}dV_{t} \ \ = \ \ C\int\nabla{\rm Ric}\ast\bigg(\nabla{\rm Rm}\ast|{\rm Rm}|^{p-2}\ast\phi^{2p}
$$
$$
+ \ {\rm Rm}\ast\nabla|{\rm Rm}|^{p-2}\ast\phi^{2p}+{\rm Rm}\ast|{\rm Rm}|^{p-2}\ast\nabla\phi^{2p}\bigg)dV_{t}
$$
$$
\leq \ C\int|\nabla{\rm Ric}||\nabla{\rm Rm}||{\rm Rm}|^{p-2}\phi^{2p}dV_{t}
+C\int|\nabla{\rm Ric}||\nabla\phi||{\rm Rm}|^{p-1}\phi^{2p-1}dV_{t}
$$
$$
\leq\int|\nabla{\rm Ric}|^{2}|{\rm Rm}|^{p-1}\phi^{2p}dV_{t}
+C\int|\nabla{\rm Rm}|^{2}|{\rm Rm}|^{p-3}\phi^{2p}dV_{t}
$$
$$
+ \ C\int|\nabla\phi|^{2}|{\rm Rm}|^{p-1}\phi^{2p-2}dV_{t}
$$
and
$$
C\int(\nabla^{2}F\ast F\ast{\rm Rm})|{\rm Rm}|^{p-2}\phi^{2p}dV_{t} \  \ = \ \ C\int\nabla F\ast\bigg(\nabla F\ast{\rm Rm}\ast|{\rm Rm}|^{p-2}\ast\phi^{2p}
$$
$$
+ \ F\ast\nabla{\rm Rm}\ast|{\rm Rm}|^{p-2}\ast\phi^{2p}+F\ast{\rm Rm}\ast\nabla|{\rm Rm}|^{p-2}\ast\phi^{2p}
$$
$$
+ \ F\ast{\rm Rm}\ast|{\rm Rm}|^{p-2}\ast\nabla\phi^{2p}\bigg)dV_{t} \ \leq \ C\int|\nabla F|^{2}|{\rm Rm}|^{p-1}\phi^{2p}dV_{t}
$$
$$
+ \ CL\int|\nabla F||\nabla{\rm Rm}||{\rm Rm}|^{p-2}\phi^{2p}dV_{t}+CL\int|\nabla F||\nabla\phi||{\rm Rm}|^{p-1}\phi^{2p-1}dV_{t}
$$
$$
\leq \ C(1+L^{2})\int|\nabla F|^{2}|{\rm Rm}|^{p-1}\phi^{2p}dV_{t}
+C\int|\nabla{\rm Rm}|^{2}|{\rm Rm}|^{p-3}\phi^{2p}dV_{t}
$$
$$
+ \ C\int|\nabla\phi|^{2}|{\rm Rm}|^{p-1}\phi^{2p-2}dV_{t}.
$$
Define {\bf bad terms} by
\begin{eqnarray*}
B_{1}&:=&\int|\nabla{\rm Ric}|^{2}|{\rm Rm}|^{p-1}\phi^{2p}dV_{t}\\
B_{2}&:=&\int|\nabla{\rm Rm}|^{2}|{\rm Rm}|^{p-3}\phi^{2p}dV_{t}\\
B_{3}&:=&\int|\nabla F|^{2}|{\rm Rm}|^{p-1}\phi^{2p}dV_{t}\\
B_{4}&:=&\int|\nabla F|^{2}|{\rm Rm}|^{p-3}\phi^{2p}dV_{t}
\end{eqnarray*}
and {\bf good terms} by
\begin{eqnarray*}
A_{1}&:=&\int|{\rm Rm}|^{p}\phi^{2p}dV_{t}\\
A_{2}&:=&\int|{\rm Rm}|^{p-1}\phi^{2p}dV_{t}\\
A_{3}&:=&\int|{\rm Rm}|^{p-1}|\nabla\phi|^{2}\phi^{2p-1}dV_{t}\\
A_{4}&:=&\int|{\rm Rm}|^{p-1}|\nabla\phi|^{2}\phi^{2p-2}dV_{t}.
\end{eqnarray*}
Hence, the above computation shows
\begin{equation}
    \frac{d}{dt}A_{1}\leq B_{1}+CB_{2}+C(1+L^{2})B_{3}
    +CA_{4}+C(K+L^{2})A_{1}.\label{2.14}
\end{equation}

%%%%%%%%%%%%%%%%%%%%%%%%%%%%%%%%%%%%%%%%%%%%%%%%%%%%%%%%%%%%%%%%%%%%%%%%%%%%%%
\subsection{The estimate for $\boldsymbol{B_{1}}$}
%%%%%%%%%%%%%%%%%%%%%%%%%%%%%%%%%%%%%%%%%%%%%%%%%%%%%%%%%%%%%%%%%%%%%%%%%%%%%%

Using (\ref{2.6}) yields
\begin{eqnarray}
B_{1}&=&\int|\nabla{\rm Ric}|^{2}|{\rm Rm}|^{p-1}\phi^{2p}dV_{t}\nonumber\\
&\leq&\int|{\rm Rm}|^{p-1}\phi^{2p}
\bigg[-\frac{1}{2}\Box|{\rm Ric}|^{2}+CK^{2}|{\rm Rm}|+CK^{2}L^{2}\nonumber\\
&&+ \ {\rm Ric}\ast(\nabla^{2}F\ast F+\nabla F\ast\nabla F)\bigg]dV_{t}\label{2.15}\\
&\leq&\frac{1}{2}\int\left[(\Delta-\partial_{t})|{\rm Ric}|^{2}\right]|{\rm Rm}|^{p-1}
\phi^{2p}dV_{t}+CK^{2} A_{1}+CK^{2}L^{2}A_{2}\nonumber\\
&&+ \ \int({\rm Ric}\ast\nabla^{2}F\ast F
+{\rm Ric}\ast\nabla F\ast\nabla F)|{\rm Rm}|^{p-1}\phi^{2p}dV_{t}.\nonumber
\end{eqnarray}
The last two integrals in (\ref{2.15}) can be estimated by
$$
\int({\rm Ric}\ast\nabla^{2}F\ast F)|{\rm Rm}|^{2p-1}\phi^{2p}dV_{t} = \int\nabla F\ast\nabla({\rm Ric}\ast F\ast|{\rm Rm}|^{p-1}\ast\phi^{2p})dV_{t}
$$
$$
= \ \int\nabla F\ast\bigg[\nabla{\rm Ric}\ast F\ast|{\rm Rm}|^{p-1}\ast\phi^{2p}
+{\rm Ric}\ast\nabla F\ast|{\rm Rm}|^{p-1}\ast\phi^{2p}
$$
$$
+ \ {\rm Ric}\ast F\ast\nabla|{\rm Rm}|^{p-1}\ast\phi^{2p}
+{\rm Ric}\ast F\ast|{\rm Rm}|^{p-1}\ast\nabla\phi^{2p}\bigg]dV_{t}
$$
$$
\leq \ CL\int|\nabla F||\nabla{\rm Ric}||{\rm Rm}|^{p-1}\phi^{2p}dV_{t}
+CK\int|\nabla F|^{2}|{\rm Rm}|^{p-1}\phi^{2p}dV_{t}
$$
$$
+ \ CLK\int|\nabla F||\nabla{\rm Rm}||{\rm Rm}|^{p-2}\phi^{2p}dV_{t}
+CLK\int|\nabla F||\nabla\phi||{\rm Rm}|^{p-1}\phi^{2p-1}dV_{t}
$$
$$
\leq \ \frac{1}{100}B_{1}+C(K+L^{2})B_{3}+CK^{2}B_{2}+CK^{2}A_{4}
$$
and
$$
\int({\rm Ric}\ast\nabla F\ast\nabla F)|{\rm Rm}|^{p-1}\phi^{2p}dV_{t}
\leq CK B_{3}.
$$
Therefore
\begin{eqnarray}
\int({\rm Ric}\ast\nabla^{2}F\ast F
+{\rm Ric}\ast\nabla F\ast\nabla F)|{\rm Rm}|^{2}\phi^{6}dV_{t}\nonumber
\end{eqnarray}
\begin{eqnarray}
&\leq&\frac{1}{100}B_{1}+C(K+L^{2})B_{3}+CK^{2}B_{2}+CK^{2}A_{4}.\label{2.16}
\end{eqnarray}
Next we compute the first integral in (\ref{2.15}). Using (\ref{2.10}) we obtain
$$
\frac{1}{2}\int[(\Delta-\partial_{t})|{\rm Ric}|^{2}]|{\rm Rm}|^{p-1}\phi^{2p}dV_{t} \ \ = \ \ \frac{1}{2}\int(\Delta|{\rm Ric}|^{2})|{\rm Rm}|^{p-1}\phi^{2p}dV_{t}
$$
$$
- \ \frac{1}{2}\int\bigg[\partial_{t}(|{\rm Ric}|^{2}|{\rm Rm}|^{p-1}\phi^{2p}dV_{t})
-|{\rm Ric}|^{2}(\partial_{t}|{\rm Rm}|^{p-1})\phi^{2p}dV_{t}
$$
$$
- \ |{\rm Ric}|^{2}|{\rm Rm}|^{p-1}\phi^{2p}(\partial_{t}dV_{t})\bigg]
$$
\begin{equation}
\leq-\frac{1}{2}
\left[\int\langle\nabla|{\rm Ric}|^{2}, \nabla|{\rm Rm}|^{p-1}
\rangle\phi^{2p}dV_{t}+\langle\nabla|{\rm Ric}|^{2}, \nabla\phi^{2p}\rangle|{\rm Rm}|^{p-1}
dV_{t}\right]\label{2.17}
\end{equation}
$$
+ \ CK^{2}(K+L^{2})A_{2}-\frac{1}{2}\frac{d}{dt}\left(\int|{\rm Ric}|^{2}|{\rm Rm}|^{p-1}
\phi^{2p}dV_{t}\right)
$$
$$
+ \ \frac{1}{2}\int|{\rm Ric}|^{2}(\partial_{t}|{\rm Rm}|^{p-1})\phi^{2p}dV_{t}.
$$
We first compute
$$
-\frac{1}{2}\int\langle\nabla|{\rm Ric}|^{2},\nabla|{\rm Rm}|^{p-1}\rangle\phi^{2p}dV_{t}
$$
\begin{equation}
= \ \ \int{\rm Ric}\ast\nabla{\rm Ric}\ast|{\rm Rm}|^{p-3}\ast{\rm Rm}
\ast\nabla{\rm Rm}\ast\phi^{2p}dV_{t}\label{2.18}
\end{equation}
$$
\leq \ \ CK\int|\nabla{\rm Ric}||\nabla{\rm Rm}||{\rm Rm}|^{p-2}\phi^{2p}dV_{t} \ \ \leq \ \ 
\frac{1}{100}B_{1}+CK^{2}B_{2}
$$
and
$$
-\frac{1}{2}\int\langle\nabla|{\rm Ric}|^{2}, \nabla\phi^{2p}\rangle|{\rm Rm}|^{p-1}dV_{t}
$$
\begin{equation}
= \ \ \int{\rm Ric}\ast\nabla{\rm Ric}\ast\phi^{2p-1}\ast\nabla\phi\ast|{\rm Rm}|^{p-1}dV_{t}\label{2.19}
\end{equation}
$$
\leq \ \ CK\int|\nabla{\rm Ric}||\nabla\phi||{\rm Rm}|^{p-1}\phi^{2p-1}dV_{t} \ \ 
\leq \ \ \frac{1}{100}B_{1}+CK^{2}A_{4}.
$$
The last integral in (\ref{2.17}) can be simplified, by using (\ref{2.8}), into
$$
\frac{1}{2}\int|{\rm Ric}|^{2}(\partial_{t}|{\rm Rm}|^{p-1})
\phi^{2p}dV_{t}=\frac{p-1}{4}\int|{\rm Ric}|^{2}|{\rm Rm}|^{p-3}
\bigg[\nabla^{2}{\rm Ric}\ast{\rm Rm}
$$
$$
+ \ {\rm Ric}\ast{\rm Rm}\ast{\rm Rm}+{\rm Rm}\ast{\rm Rm}\ast F\ast F+{\rm Rm}\ast\nabla^{2}F\ast F+{\rm Rm}\ast\nabla F\ast\nabla F
\bigg]\phi^{2p}dV_{t}
$$
$$
\leq C\int(\nabla^{2}{\rm Ric}\ast{\rm Rm})|{\rm Ric}|^{2}|{\rm Rm}|^{p-3}\phi^{2p}dV_{t}
+CK^{2}A_{1}+CKL^{2}A_{1}
$$
$$
+ \ C\int({\rm Rm}\ast\nabla F\ast \nabla F+{\rm Rm}\ast\nabla^{2} F\ast F)|{\rm Ric}|^{2}|{\rm Rm}|^{p-3}
\phi^{2p}dV_{t}.
$$
According to
$$
C\int(\nabla^{2}{\rm Ric}\ast{\rm Rm})|{\rm Ric}|^{2}|{\rm Rm}|^{p-3}
\phi^{2p}dV_{t}
$$
$$
= \ \ C\int\nabla{\rm Ric}\ast\nabla\left(|{\rm Ric}|^{2}
\phi^{2p}\ast{\rm Rm}\ast|{\rm Rm}|^{p-3}\right)dV_{t}
$$
$$
= \ \ C\int\nabla{\rm Ric}\ast\bigg[{\rm Ric}\ast\nabla{\rm Ric}
\ast{\rm Rm}\ast\phi^{2p}|{\rm Rm}|^{p-3}+|{\rm Ric}|^{2}|{\rm Rm}|^{p-3}\phi^{2p-1}\ast\nabla\phi\ast{\rm Rm}
$$
$$
+ \ |{\rm Ric}|^{2}\phi^{2p}\ast\nabla{\rm Rm}\ast|{\rm Rm}|^{p-3}+|{\rm Ric}|^{2}\phi^{2p}
\ast{\rm Rm}\ast|{\rm Rm}|^{p-5}\ast{\rm Rm}\ast\nabla{\rm Rm}\bigg]dV_{t}
$$
$$
\leq \ \ CK\int|\nabla{\rm Ric}|^{2}|{\rm Rm}|^{p-2}\phi^{2p}dV_{t}
+C\int|\nabla{\rm Ric}||\nabla{\rm Rm}||{\rm Ric}|^{2}|{\rm Rm}|^{p-3}\phi^{2p}dV_{t}
$$
$$
+ \ C\int|\nabla{\rm Ric}||\nabla\phi||{\rm Ric}|^{2}|{\rm Rm}|^{p-2}\phi^{2p-1}dV_{t} \ \ \leq \ \ \frac{1}{100}B_{1}+CK^{2}B_{2}+CK^{2}A_{4}
$$
and
$$
C\int({\rm Rm}\ast\nabla F\ast\nabla F)|{\rm Ric}|^{2}||{\rm Rm}|^{p-3}\phi^{2p}dV_{t} \ \ \leq \ \ C\int|\nabla F|^{2}|{\rm Ric}|^{2}|{\rm Rm}|^{p-2}\phi^{2p}dV_{t}
$$
$$
\leq \ \ CK\int|\nabla F|^{2}|{\rm Rm}|^{p-1}\phi^{2p}dV_{t} \ \ \leq \ \ CKB_{3}
$$
and
$$
C\int({\rm Rm}\ast\nabla^{2}F\ast F)|{\rm Ric}|^{2}|{\rm Rm}|^{p-3}\phi^{2p}dV_{t}
$$
$$
= \ \ C\int\nabla F\ast\nabla({\rm Rm}\ast F\ast|{\rm Ric}|^{2}\ast|{\rm Rm}|^{p-3}\ast\phi^{2p})dV_{t}
$$
$$
= \ \ C\int\nabla F\ast\bigg[\nabla{\rm Rm}\ast F\ast|{\rm Ric}|^{2}\ast|{\rm Rm}|^{p-3}\ast\phi^{2p}
+{\rm Rm}\ast\nabla F\ast|{\rm Ric}|^{2}\ast|{\rm Rm}|^{p-3}\ast\phi^{2p}
$$
$$
+ \ {\rm Rm}\ast F\ast{\rm Ric}\ast\nabla{\rm Ric}\ast|{\rm Rm}|^{p-3}\ast\phi^{2p}
+{\rm Rm}\ast F\ast|{\rm Ric}|^{2}\ast|{\rm Rm}|^{p-3}\ast\phi^{2p-1}\ast\nabla\phi
$$
$$
+ \ {\rm Rm}\ast F\ast|{\rm Ric}|^{2}\ast\phi^{2p}\ast|{\rm Rm}|^{p-5}\ast{\rm Rm}\ast\nabla{\rm Rm}\bigg]dV_{t}
$$
$$
\leq \ \ CLK\int|\nabla F||\nabla{\rm Rm}||{\rm Rm}|^{p-2}\phi^{2p}dV_{t}
+CK\int|\nabla F|^{2}|{\rm Rm}|^{p-1}\phi^{2p}dV_{t}
$$
$$
+ \ CL\int|\nabla F||\nabla{\rm Ric}||{\rm Rm}|^{p-1}\phi^{2p}dV_{t}
+CLK\int|\nabla F||\nabla\phi||{\rm Rm}|^{p-1}\phi^{2p-1}dV_{t}
$$
$$
\leq \ \ \frac{1}{100}B_{1}+CK^{2} B_{2}+C(K+L^{2})B_{3}
+CK^{2}A_{4},
$$
and plugging (\ref{2.18}) -- (\ref{2.19}) into (\ref{2.17}), we arrive at
$$
\frac{1}{2}\int[(\Delta-\partial_{t})|{\rm Ric}|^{2}]|{\rm Rm}|^{2}\phi^{6}dV_{t} \ \ 
\leq \ \ -\frac{d}{dt}\left(\frac{1}{2}\int|{\rm Ric}|^{2}|{\rm Rm}|^{p-1}\phi^{2p}dV_{t}\right)
$$
\begin{equation}
+ \ \ \frac{4}{100}B_{1}+CK^{2}B_{2}+C(K+L^{2})B_{3}+CK^{2}(K+L^{2})A_{2}+CK^{2}A_{4}.\label{2.20}
\end{equation}

Substituting (\ref{2.16}) and (\ref{2.20}) into (\ref{2.15}) yields

\begin{lemma}\label{l2.2} We have
\begin{eqnarray}
B_{1}&\leq& CKB_{2}+C(K+L^{2})B_{3}+CKA_{1}+CK(K+L^{2})A_{2}+CKA_{4}\nonumber\\
&&+ \ CK(K+L^{2})A_{1}-\frac{d}{dt}\left(\frac{1}{2}\int|{\rm Ric}|^{2}|{\rm Rm}|^{p-1}\phi^{2p}dV_{t}\right).\label{2.21}
\end{eqnarray}
\end{lemma}

%%%%%%%%%%%%%%%%%%%%%%%%%%%%%%%%%%%%%%%%%%%%%%%%%%%%%%%%%%%%%%%%%%%%%%%%%%%%%%
\subsection{The estimate for $\boldsymbol{B_{2}}$}\label{subsection2.2}
%%%%%%%%%%%%%%%%%%%%%%%%%%%%%%%%%%%%%%%%%%%%%%%%%%%%%%%%%%%%%%%%%%%%%%%%%%%%%%

Using (\ref{2.7}) we obtain
\begin{eqnarray}
B_{2}&=&\int|\nabla{\rm Rm}|^{2}|{\rm Rm}|^{p-3}\phi^{2p}dV_{t} \ \ \leq \ \ \int|{\rm Rm}|^{p-3}\phi^{2p}\bigg[-\frac{1}{2}\Box|{\rm Rm}|^{2}
+C|{\rm Rm}|^{3}\nonumber\\
&&+ \ CL^{2}|{\rm Rm}|^{2}+{\rm Rm}\ast(\nabla^{2}F\ast F+\nabla F\ast\nabla F)\bigg]dV_{t}\nonumber\\
&=&\frac{1}{2}\int[(\Delta-\partial_{t})|{\rm Rm}|^{2}]|{\rm Rm}|^{p-3}\phi^{2p}dV_{t}+CA_{1}+CL^{2}A_{2}\label{2.22}\\
&&+ \ \int({\rm Rm}\ast\nabla^{2}F\ast F+{\rm Rm}\ast\nabla F\ast\nabla F)|{\rm Rm}|^{p-3}\phi^{2p}dV_{t}.\nonumber
\end{eqnarray}
To deal with the last two integrals in (\ref{2.22}), we need to estimate the term
$$
\int|\nabla F|^{2}|{\rm Rm}|^{p-2}\phi^{2p}dV_{t}
$$
that can be bounded by $B_{3}$ and $B_{4}$. Indeed
\begin{eqnarray}
\int|\nabla F|^{2}|{\rm Rm}|^{p-2}\phi^{2p}dV_{t}
&=&\int\left(|\nabla F||{\rm Rm}|^{\frac{p-1}{2}}\phi^{p}\right)
\left(|\nabla F||{\rm Rm}|^{\frac{p-3}{2}}\phi^{p}\right)dV_{t}\nonumber\\
&\leq&\epsilon B_{3}+C_{\epsilon}B_{4},\label{2.23}
\end{eqnarray}
for any given $\epsilon>0$ with $C_{\epsilon}=C/\epsilon$. Using (\ref{2.23}) we obtain
$$
\int({\rm Rm}\ast\nabla^{2}F\ast F)|{\rm Rm}|^{p-3}\phi^{2p}dV_{t} \ \ = \ \ 
\int\nabla F\ast\nabla({\rm Rm}\ast F\ast|{\rm Rm}|^{p-3}\phi^{2p})dV_{t}
$$
$$
= \ \ \int\nabla F\ast\bigg[\nabla{\rm Rm}\ast F\ast|{\rm Rm}|^{p-3}\phi^{2p}
+{\rm Rm}\ast\nabla F\ast|{\rm Rm}|^{p-3}\phi^{2p}
$$
$$
+ \ {\rm Rm}\ast F\ast\phi^{2p-1}\ast\nabla\phi\ast|{\rm Rm}|^{p-3}+{\rm Rm}\ast F\ast\phi^{2p}\ast|{\rm Rm}|^{p-5}\ast{\rm Rm}\ast\nabla{\rm Rm}\bigg]dV_{t}
$$
\begin{equation}
\leq \ \ CL\int|\nabla F||\nabla{\rm Rm}||{\rm Rm}|^{p-3}\phi^{2p}dV_{t}
+C\int|\nabla F|^{2}|{\rm Rm}|^{p-2}\phi^{2p}dV_{t}\label{2.24}
\end{equation}
$$
+ \ CL\int|\nabla F||\nabla\phi||{\rm Rm}|^{p-2}\phi^{2p-1}dV_{t} \ \ \leq \ \ \frac{1}{100}B_{2}+CL^{2}B_{4}+\epsilon B_{3}+CA_{4}+C_{\epsilon}B_{4}
$$
and
\begin{equation}
\int({\rm Rm}\ast\nabla F\ast\nabla F)
|{\rm Rm}|^{p-3}\phi^{2p}dV_{t}\leq \epsilon B_{3}+C_{\epsilon}B_{4}.\label{2.25}
\end{equation}
On the other hand, we get, because $p\geq5$,
$$
\frac{1}{2}\int(\Delta|{\rm Rm}|^{2})|{\rm Rm}|^{p-3}\phi^{2p}dV_{t} \ \ = \ \ -\frac{1}{2}
\int\langle\nabla|{\rm Rm}|^{2},\nabla(|{\rm Rm}|^{p-3}\phi^{2p})\rangle dV_{t}
$$
$$
= \ -\frac{p-3}{4}\int\left|\nabla|{\rm Rm}|^{2}\right|^{2}|{\rm Rm}|^{p-5}\phi^{2p} dV_{t}
+\int\left({\rm Rm}\ast\nabla{\rm Rm}\ast\nabla\phi\ast\phi^{2p-1}|{\rm Rm}|^{p-3}\right)dV_{t}
$$
\begin{equation}
\leq \ \ C\int|\nabla{\rm Rm}||\nabla\phi||{\rm Rm}|^{p-2}\phi^{2p-1}dV_{t} \ \ \leq \ \ \frac{1}{100}B_{2}+CA_{4}.\label{2.26}
\end{equation}
Moreover
$$
-\frac{1}{2}\int(\partial_{t}|{\rm Rm}|^{2})|{\rm Rm}|^{p-3}\phi^{2p}dV_{t} \ \ = \ \ 
-\frac{1}{2}\int\bigg[\partial_{t}(|{\rm Rm}|^{2}|{\rm Rm}|^{p-3}\phi^{2p}dV_{t})
$$
$$
- \ |{\rm Rm}|^{2}(\partial_{t}|{\rm Rm}|^{p-3})\phi^{2p}dV_{t}
-|{\rm Rm}|^{2}|{\rm Rm}|^{p-3}\phi^{2p}(\partial_{t}dV_{t})\Bigg]
$$
$$
= \ -\frac{1}{2}\frac{d}{dt}\left(\int|{\rm Rm}|^{p-1}\phi^{2p}dV_{t}\right)
+\frac{1}{2}\int|{\rm Rm}|^{p-1}\phi^{2p}(\partial_{t}dV_{t})
$$
$$
+ \ \frac{p-3}{4}\int|{\rm Rm}|^{p-3}(\partial_{t}|{\rm Rm}|^{2})\phi^{2p}dV_{t}
$$
and hence
\begin{equation}
-\frac{p-1}{4}\int(\partial_{t}|{\rm Rm}|^{2})\phi^{6}dV_{t}
\leq C(K+L^{2})A_{2}-\frac{1}{2}\frac{d}{dt}\left(\int|{\rm Rm}|^{2}\phi^{6}dV_{t}\right).\label{2.27}
\end{equation}
Plugging (\ref{2.23}) -- (\ref{2.27}) into (\ref{2.22}) yields

\begin{lemma}\label{l2.3} We have
\begin{eqnarray}
B_{2}&\leq& \epsilon B_{3}+C(C_{\epsilon}+L^{2})B_{4}+CA_{1}+C(K+L^{2})A_{2}+CA_{4}\nonumber\\
&&- \ \frac{d}{dt}\left(\frac{2}{p-1}\int|{\rm Rm}|^{p-1}\phi^{2p}dV_{t}\right).\label{2.28}
\end{eqnarray}
\end{lemma}

%%%%%%%%%%%%%%%%%%%%%%%%%%%%%%%%%%%%%%%%%%%%%%%%%%%%%%%%%%%%%%%%%%%%%%%%%%%%%%
\subsection{The estimate for $\boldsymbol{B_{3}}$}\label{subsection2.3}
%%%%%%%%%%%%%%%%%%%%%%%%%%%%%%%%%%%%%%%%%%%%%%%%%%%%%%%%%%%%%%%%%%%%%%%%%%%%%%

According to (\ref{2.11}), we have
$$
B_{3} \ = \ \int|\nabla F|^{2}|{\rm Rm}|^{p-1}\phi^{2p}dV_{t} \ \leq \ 
\int|{\rm Rm}|^{p-1}\phi^{2p}
\left[-\Box|F|^{2}+CL^{2}|{\rm Rm}|+CL^{4}\right]dV_{t}
$$
$$
\leq \ \int|{\rm Rm}|^{p-1}\phi^{2p}
[(\Delta-\partial_{t})|F|^{2}]dV_{t}+CL^{2}A_{1}+CL^{4}A_{2}.
$$
First compute
$$
\int(\Delta |F|^{2})|{\rm Rm}|^{p-1}\phi^{2p}dV_{t} \ \ = \ \ \int\nabla|F|^{2}\ast\nabla(|{\rm Rm}|^{p-1}\ast\phi^{2p})dV_{t}
$$
$$
= \ \ \int F\ast\nabla F\ast(|{\rm Rm}|^{p-3}\ast{\rm Rm}\ast\nabla{\rm Rm}\ast\phi^{2p}
+|{\rm Rm}|^{p-1}\ast\phi^{2p-1}\ast\nabla\phi)dV_{t}
$$
$$
\leq \ \ CL\int|\nabla F||\nabla{\rm Rm}||{\rm Rm}|^{p-2}\phi^{2p}dV_{t}+CL\int|\nabla F||\nabla\phi||{\rm Rm}|^{p-1}\phi^{2p-1}dV_{t}
$$
$$
\leq \ \ \frac{1}{100}B_{3}+CL^{2}B_{2}+CL^{2}A_{4}.
$$
Using $\partial_{t}|{\rm Rm}|^{p-1}=\frac{p-1}{2}|{\rm Rm}|^{p-3}\partial_{t}|{\rm Rm}|^{2}$ yields
$$
-\int|{\rm Rm}|^{p-1}\phi^{2p}
(\partial_{t}|F|^{2})dV_{t} \ \ = \ \ -\int\bigg[\partial_{t}(|{\rm Rm}|^{p-1}\phi^{2p}
|F|^{2}dV_{t})
$$
$$
- \ (\partial_{t}|{\rm Rm}|^{p-1})\phi^{2p}|F|^{2}dV_{t}
-|{\rm Rm}|^{p-1}\phi^{2p}|F|^{2}(\partial_{t}dV_{t})\bigg]
$$
$$
= \ \ -\frac{d}{dt}\left(\int|F|^{2}|{\rm Rm}|^{p-1}\phi^{2p}dV_{t}\right)
+\int\phi^{2p}|F|^{2}(\partial_{t}|{\rm Rm}|^{p-1})dV_{t}
$$
$$
+ \ \int|{\rm Rm}|^{p-1}|F|^{2}\phi^{2p}(\partial_{t}dV_{t}) \ \ \leq \ \ C(K+L^{2})L^{2}A_{2}
-\frac{d}{dt}\left(\int|F|^{2}|{\rm Rm}|^{p-1}\phi^{2p}dV_{t}\right)
$$
$$
+ \ \int\phi^{2p}|F|^{2}|{\rm Rm}|^{p-3}\bigg(\nabla^{2}{\rm Ric}\ast{\rm Rm}
+{\rm Ric}\ast{\rm Rm}\ast{\rm Rm}+{\rm Rm}\ast{\rm Rm}\ast F\ast F
$$
$$
+ \ {\rm Rm}\ast\nabla^{2}F\ast F+{\rm Rm}\ast\nabla F\ast\nabla F\bigg)dV_{t}
$$
$$
\leq \ \ C(K+L^{2})L^{2}A_{2}
-\frac{d}{dt}\left(\int|F|^{2}|{\rm Rm}|^{p-1}\phi^{2p}dV_{t}\right)
$$
$$
+ \ \int(\nabla^{2}{\rm Ric}\ast{\rm Rm})|{\rm Rm}|^{p-3}|F|^{2}\phi^{2p}dV_{t}+ CKL^{2}A_{2}+CL^{4}A_{2}
$$
$$
+ \ \int(\nabla^{2}F\ast F\ast{\rm Rm})|{\rm Rm}|^{p-3}|F|^{2}\phi^{2p}dV_{t}
+CL^{2}\int|\nabla F|^{2}|{\rm Rm}|^{p-2}\phi^{2p}dV_{t}.
$$
From (\ref{2.25}), we can get
$$
CL^{2}\int|\nabla F|^{2}|{\rm Rm}|^{p-2}\phi^{2p}dV_{t}
\leq CL^{2}(\epsilon B_{3}+C_{\epsilon}B_{4})
\leq\frac{1}{100}B_{3}+CL^{4}B_{4}
$$
by choosing $\epsilon=1/100 CL^{2}$. Moreover
$$
\int(\nabla^{2}{\rm Ric}\ast{\rm Rm})|{\rm Rm}|^{p-3}|F|^{2}\phi^{2p}dV_{t} \ \ = \ \ 
\int\nabla{\rm Ric}\ast\nabla({\rm Rm}\ast|F|^{2}\ast|{\rm Rm}|^{p-3}\ast\phi^{2p})dV_{t}
$$
$$
\leq \ \ CL^{2}\int|\nabla{\rm Ric}||\nabla{\rm Rm}||{\rm Rm}|^{p-3}\phi^{2p}dV_{t}+CL\int|\nabla{\rm Ric}||\nabla F||{\rm Rm}|^{p-2}
\phi^{2p}dV_{t}
$$
$$
+ \ CL^{2}\int|\nabla{\rm Ric}||\nabla{\rm Rm}||{\rm Rm}|^{p-3}\phi^{2p}dV_{t}
+CL^{2}\int|\nabla{\rm Ric}||\nabla\phi||{\rm Rm}|^{p-2}\phi^{2p-1}dV_{t}
$$
$$
\leq \ \ CL^{2}B_{2}+\frac{1}{100}B_{3}+CL^{2}A_{4}.
$$
Finally,  using again (\ref{2.25}), we arrive at
$$
\int(\nabla^{2}F\ast F\ast{\rm Rm})|{\rm Rm}|^{p-3}|F|^{:2}\phi^{2p}dV_{t} \ \ = \ \ \int\nabla F\ast\nabla(
F\ast{\rm Rm}\ast|{\rm Rm}|^{p-3}|F|^{2}\phi^{2p})dV_{t}
$$
$$
\leq \ \ CL^{2}\int|\nabla F|^{2}|{\rm Rm}|^{p-2}\phi^{2p}dV_{t}
+CL^{3}\int|\nabla F||\nabla{\rm Rm}||{\rm Rm}|^{p-3}\phi^{2p}dV_{t}
$$
$$
+ \ CL^{2}\int|\nabla F|^{2}|{\rm Rm}|^{p-2}\phi^{2p}dV_{t}
+CL^{3}\int|\nabla F||\nabla\phi||{\rm Rm}|^{p-2}\phi^{2p-1}dV_{t}
$$
$$
\leq \ \ \frac{1}{100}B_{3}+CL^{4}B_{4}+CL^{2}B_{2}
+CL^{4}B_{4}+\frac{1}{100}B_{3}+CL^{6}A_{4}.
$$
Consequently, 
$$
B_{3} \ \ \leq \ \ \frac{5}{100}B_{3}-\frac{d}{dt}\left(\int|F|^{2}|{\rm Rm}|^{p-1}\phi6{2p}dV_{t}\right)
$$
$$
+ \ C(K+L^{2})L^{2}A_{2}+CL^{2}B_{2}+CL^{4}B_{4}+CL^{2}(1+L^{4})A_{4}+CL^{2}A_{2}
$$
so that
\begin{eqnarray}
B_{3}&\leq& -\frac{d}{dt}\left(2\int|F|^{2}|{\rm Rm}|^{p-1}\phi^{2p}dV_{t}\right)
+CL^{2}B_{2}+CL^{4}B_{4}\nonumber\\
&&+ \ C(K+L^{2})L^{2}A_{2}+CL^{2}A_{1}+CL^{2}(1+L^{4})A_{4}.\label{2.29}
\end{eqnarray}
Taking $\epsilon=1/2CL^{2}$ in (\ref{2.28}) and using (\ref{2.29}) we obtain
\begin{eqnarray*}
B_{2}&\leq&-\frac{d}{dt}\left(\frac{1}{CL^{2}}\int|F|^{2}|{\rm Rm}|^{2p-1}\phi^{2p}dV_{t}\right)+\frac{1}{2}B_{2}+CL^{2}B_{4}+C(K+L^{2})A_{2}\nonumber\\
&&+ \ CA_{1}+C(1+L^{4})A_{4}-\frac{d}{dt}
\left(\frac{2}{p-1}\int|{\rm Rm}|^{p-1}\phi^{2p}dV_{t}\right).
\end{eqnarray*}
Thus

\begin{lemma}\label{l2.4} We have
\begin{eqnarray}
B_{2}&\leq&-\frac{d}{dt}\left(\frac{C}{L^{2}}\int|{\rm Rm}|^{p-1}|F|^{2}\phi^{2p}dV_{t}
+C\int|{\rm Rm}|^{p-1}\phi^{2p}dV_{t}\right)\nonumber\\
&&+ \ CL^{2}B_{4}+CA_{1}+C(K+L^{2})A_{2}+C(1+L^{4})A_{4}\label{2.30}
\end{eqnarray}
and
\begin{eqnarray}
B_{3}&\leq&-\frac{d}{dt}\left(C\int|{\rm Rm}|^{p-1}|F|^{2}\phi^{2p}dV_{t}
+C(1+L^{2})\int|{\rm Rm}|^{p-1}\phi^{2p}dV_{t}\right)+CL^{4}B_{4}\nonumber\\
&&+ \ C(1+L^{2})A_{1}+C(1+L^{2})(K+L^{2})A_{2}
+C(1+L^{2})(1+L^{4})A_{4}.\label{2.31}
\end{eqnarray}
\end{lemma}

%%%%%%%%%%%%%%%%%%%%%%%%%%%%%%%%%%%%%%%%%%%%%%%%%%%%%%%%%%%%%%%%%%%%%%%%%%%%%%
\subsection{The evolution for $\boldsymbol{A_{1}}$: {\bf I}}\label{subsection2.4}
%%%%%%%%%%%%%%%%%%%%%%%%%%%%%%%%%%%%%%%%%%%%%%%%%%%%%%%%%%%%%%%%%%%%%%%%%%%%%%

To simplify notions, we introduce
\begin{equation}
R:=1+\sum_{1\leq i\leq 6}(K^{i}+L^{i}).\label{2.32}
\end{equation}
Therefore, from (\ref{2.21}), (\ref{2.30}) and (\ref{2.31}) we obtain
\begin{eqnarray}
B_{1}&\leq&CR(B_{2}+B_{3})+CR(A_{1}+A_{2}+A_{4})\nonumber\\
&&- \ \frac{d}{dt}
\left(C\int|{\rm Ric}|^{2}|{\rm Rm}|^{p-1}
\phi^{2p}dV_{t}\right),\label{2.33}\\
B_{2}&\leq&CRB_{4}+CR(A_{1}+A_{2}+A_{4})\nonumber\\
&&- \ \frac{d}{dt}\left(\frac{C}{L^{2}}\int|{\rm Rm}|^{p-1}|F|^{2}\phi^{2p}dV_{t}
+C\int|{\rm Rm}|^{p-1}\phi^{2p}dV_{t}\right),\label{2.34}\\
B_{3}&\leq&CRB_{4}+CR(A_{1}+A_{2}+A_{4})\nonumber\\
&&- \ \frac{d}{dt}\left(C\int|{\rm Rm}|^{p-1}|F|^{2}\phi^{2p}dV_{t}
+C(1+L^{2})\int|{\rm Rm}|^{p-1}\phi^{2p}dV_{t}\right).\label{2.35}
\end{eqnarray}
Moreover, the inequality (\ref{2.14}) becomes
\begin{equation}
    \frac{d}{dt}A_{1}\leq CR(B_{1}+B_{2}+B_{3})
    +CR(A_{1}+A_{4}).\label{2.36}
\end{equation}
Plugging (\ref{2.33})--(\ref{2.35}) into (\ref{2.36}) yields
\begin{eqnarray}
\frac{d}{dt}A_{1}&\leq&-\frac{d}{dt}\bigg[\frac{C(1+L^{2})R(1+R)}{L^{2}}\int|{\rm Rm}|^{p-1}|F|^{2}\phi^{2p}dV_{t}\nonumber\\
&&+ \ C(1+L^{2})R(1+R)\int|{\rm Rm}|^{p-1}\phi^{2p}dV_{t}+CR\int|{\rm Ric}|^{2}|{\rm Rm}|^{p-1}\phi^{2p}dV_{t}\bigg]\label{2.37}\\
&&+ \ CR^{2}(1+R)B_{4}
+CR^{2}(1+R)(A_{1}+A_{2}+A_{4}).\nonumber
\end{eqnarray}

%%%%%%%%%%%%%%%%%%%%%%%%%%%%%%%%%%%%%%%%%%%%%%%%%%%%%%%%%%%%%%%%%%%%%%%%%%%%%%
\subsection{The estimate for $\boldsymbol{B_{4}}$}\label{subsection2.5}
%%%%%%%%%%%%%%%%%%%%%%%%%%%%%%%%%%%%%%%%%%%%%%%%%%%%%%%%%%%%%%%%%%%%%%%%%%%%%%

Recall that
$$
B_{3}=\int|\nabla F|^{2}|{\rm Rm}|^{p-1}\phi^{2p}dV_{t}, \ \ \ 
B_{4}=\int|\nabla F|^{2}|{\rm Rm}|^{p-3}\phi^{2p}dV_{t}, \ \ \ p\geq3.
$$
To estimate $B_{4}$, write
\begin{equation}
B(p):=\int|\nabla F|^{2}|{\rm Rm}|^{p-1}\phi^{2p}dV_{t}, \ \ \ p\geq 1 \ \text{and} \ p\in\mathbb{N}.\label{2.38}
\end{equation}
Using the inequality 
$$
B_{4}=\int|\nabla F|^{2}|{\rm Rm}|^{(p-2)-1}\phi^{2(p-2)+4}dV_{t}
\leq C B(p-2)
$$
the estimate (\ref{2.35}) can be written as, where $p$ is an integer greater than $3$,
\begin{equation}
B(p)\leq CR B(p-2)+Cf(p),
\end{equation}
wuith
\begin{eqnarray}
f(p)&:=&R\left(\int|{\rm Rm}|^{p}\phi^{2p}dV_{t}
+\int|{\rm Rm}|^{p-1}\phi^{2p}dV_{t}
+\int|{\rm Rm}|^{p-1}|\nabla\phi|^{2}\phi^{2p-2}dV_{t}\right)\nonumber \\
&&- \ \frac{d}{dt}\left(\int|{\rm Rm}|^{p-1}|F|^{2}\phi^{2p}dV_{t}
+(1+L^{2})\int|{\rm Rm}|^{p-1}
\phi^{2p}dV_{t}\right).\label{2.40}
\end{eqnarray}

\begin{lemma}\label{l2.5} For any integer $p\geq3$ one has
\begin{equation}
    B(2m+1)\leq (CR)^{m}B(1)+\sum_{1\leq i\leq m}C(CR)^{m-i}f(2i+1), \ \ \ p=2m+1,
    \label{2.41}
\end{equation}
and
\begin{equation}
B(2m+2)\leq(CR)^{m}B(2)+\sum_{1\leq i\leq m}C(CR)^{m-i}f(2i+2), \ \ \ p=2m+2.\label{2.42}
\end{equation}
\end{lemma}

\begin{proof} For $p=3,4,5,6$ we obtain
\begin{eqnarray*}
B(3)&\leq&CRB(1)+Cf(3),\\
B(5)&\leq&CR B(3)+Cf(5) \ \ \leq \ \ (CR)^{2}B(1)
+Cf(5)+C^{2}Rf(3),\\
B(4)&\leq&CR B(2)+Cf(4),\\
B(6)&\leq& CR B(4)+Cf(6) \ \ \leq \ \ (CR)^{2}B(2)+Cf(6)+C^{2}R f(4).
\end{eqnarray*}
In general, by induction on $m$ we easily get the desired result.
\end{proof}

Because the integral 
$$
B(1)=\int|\nabla F|^{2}\phi^{2}dV_{t}
$$
does not involve any curvature, we shall use (\ref{2.41}), where $p$ is odd, to estimate $B_{4}$.

\begin{lemma}\label{l2.6} One has
\begin{eqnarray}
B(1)&\leq& CR\left(\int|{\rm Rm}|\phi^{2}dV_{t}+\int\phi^{2}dV_{t}
+\int|\nabla\phi|^{2}dV_{t}\right)\nonumber\\
&&- \ \frac{d}{dt}\left(2\int|F|^{2}\phi^{2}dV_{t}\right).\label{2.43}
\end{eqnarray}
\end{lemma}

\begin{proof} For convenience let
\begin{eqnarray*}
B_{1}(p)&:=&\int|\nabla{\rm Ric}|^{2}|{\rm Rm}|^{p-1}\phi^{2p}dV_{t}, \ \ \ B_{2}(p) \ \ := \ \ \int|\nabla{\rm Rm}|^{2}|{\rm Rm}|^{p-3}\phi^{2p}dV_{t}\\
B_{3}(p)&:=&\int|\nabla F|^{2}|{\rm Rm}|^{p-1}\phi^{2p}dV_{t}, \ \ \ B_{4}(p) \ \ := \ \ \int|\nabla F|^{2}|{\rm Rm}|^{p-3}\phi^{2p}dV_{t}
\end{eqnarray*}
and
\begin{eqnarray*}
A_{1}(p)&:=&\int|{\rm Rm}|^{p}\phi^{2p}dV_{t}, \ \ \ A_{2}(p) \ \ := \ \ \int|{\rm Rm}|^{p-1}\phi^{2p}dV_{t}\\
A_{3}(p)&:=&\int|{\rm Rm}|^{p-1}|\nabla\phi|^{2}\phi^{2p-1}dV_{t}, \ \ \ A_{4}(p) \ \ := \ \ \int|{\rm Rm}|^{p-1}|\nabla\phi|^{2}\phi^{2p-2}dV_{t}.
\end{eqnarray*}
Hence (\ref{2.43}) is equivalent to
\begin{equation}
    B(1)\leq CR(A_{1}(1)+A_{2}(1)+A_{4}(1)).\label{2.44}
\end{equation}
According to (\ref{2.11}), we arrive at
$$
B(1) \ \ = \ \ \int|\nabla F|^{2}\phi^{2}dV_{t} \ \ 
\leq \ \ \int\left[(\Delta-\partial_{t})|F|^{2}+CL^{2}|{\rm Rm}|+CL^{4}\right]\phi^{2}dV_{t}
$$
$$
\leq \ \ -\int\langle\nabla|F|^{2},\nabla\phi^{2}\rangle dV_{t}
-\int(\partial_{t}|F|^{2})\phi^{2}dV_{t}
+CR\int|{\rm Rm}|\phi^{2}dV_{t}+CR\int\phi^{2}dV_{t}
$$
$$
\leq \ \ \int(\nabla F\ast F\ast\phi\ast\nabla\phi)dV_{t}
+CRA_{1}(1)+CRA_{2}(1)
$$
$$
- \ \int\left[\partial_{t}(|F|^{2}\phi^{2}dV_{t})
-|F|^{2}\phi^{2}\partial_{t}dV_{t}\right].
$$
By (\ref{2.10}), we have
$$
\int|F|^{2}\phi^{2}\partial_{t}dV_{t}
\leq CR\int\phi^{2}dV_{t}=CRA_{2}(1)
$$
so that
$$
B(1) \ \ \leq \ \ C\int(|\nabla F|\phi)(|F||\nabla\phi|)dV_{t}
+CRA_{1}(1)+CRA_{2}(1)
-\frac{d}{dt}\left(\int|F|^{2}\phi^{2}dV_{t}\right)
$$
$$
\leq \ \ \frac{1}{2}\int|\nabla F|^{2}\phi^{2}dV_{t}
+CR(A_{1}(1)+A_{2}(1)+A_{4}(1))-\frac{d}{dt}
\left(\int|F|^{2}\phi^{2}dV_{t}\right)
$$
which implies (\ref{2.43}).
\end{proof}

Introduce
\begin{eqnarray}
D_{1}(p)&:=&\int|{\rm Rm}|^{p-1}|F|^{2}\phi^{2p}dV_{t},\label{2.45}\\
D_{2}(p)&:=&\int|{\rm Ric}|^{2}|{\rm Rm}|^{p-1}\phi^{2p}dV_{t}.\label{2.46}
\end{eqnarray}
Substituting (\ref{2.45}) into (\ref{2.40}) yields
\begin{equation}
    f(p)\leq R[A_{1}(p)+A_{2}(p)+A_{4}(p)]
    -\frac{d}{dt}[D_{1}(p)+(1+L^{2})A_{2}(p)].\label{2.47}
\end{equation}
Together (\ref{2.41}), (\ref{2.43}) and (\ref{2.47}), where $p=2m+1$ is odd, we get
\begin{eqnarray}
B(2m+1)&\leq&(CR)^{m}B(1)+\sum_{1\leq i\leq m}C(CR)^{i}f(2i+1)\nonumber\\
&\leq&C(CR)^{m}\left[R(A_{1}(1)+A_{2}(1)+A_{4}(1))
-\frac{d}{dt}D_{1}(1)\right]\nonumber\\
&&+ \ \sum_{1\leq i\leq m}C(CR)^{i}
\bigg[R(A_{1}(2i+1)+A_{2}(2i+1)+A_{4}(2i+1))\label{2.48}\\
&&- \ \frac{d}{dt}(D_{1}(2i+1)+(1+L^{2})A_{2}(2i+1))\bigg].\nonumber
\end{eqnarray}

%%%%%%%%%%%%%%%%%%%%%%%%%%%%%%%%%%%%%%%%%%%%%%%%%%%%%%%%%%%%%%%%%%%%%%%%%%%%%%
\subsection{The evolution for $\boldsymbol{A_{1}}$: {\bf II}}\label{subsection2.6}
%%%%%%%%%%%%%%%%%%%%%%%%%%%%%%%%%%%%%%%%%%%%%%%%%%%%%%%%%%%%%%%%%%%%%%%%%%%%%%

As in \cite{KMW, L1, LY}, we choose 
\begin{equation}
\Omega=B_{g}(x_{0},\rho/\sqrt{K}), \ \ \ 
\phi=\left(\frac{\rho/\sqrt{K}-d_{g}(x_{0},\cdot)}{\rho/\sqrt{K}}\right)_{+}\label{2.49}
\end{equation}
so that 
\begin{equation}
e^{-CKt}g\leq g(t)\leq e^{CKt}g, \ \ \ |\nabla_{g(t)}
\phi|_{g(t)}\leq e^{CKT}|\nabla_{g}\phi|_{g}
\leq\frac{\sqrt{K}e^{CKT}}{\rho}.\label{2.50}
\end{equation}

\begin{lemma}\label{l2.7} Under the condition (\ref{2.49}), we have
\begin{eqnarray}
    A_{1}(1)+A_{2}(1)+A_{4}(1)
   &\leq&\frac{1}{3}A_{1}(3)\nonumber\\
   &&+ \ C\left(1+\frac{K}{\rho^{2}}e^{CKT}\right){\rm Vol}_{g(t)}
    \left(B_{g}(x_{0},\rho/\sqrt{K})\right).\label{2.51}
\end{eqnarray}
\end{lemma}

\begin{proof} By Young's inequality, we have
\begin{eqnarray*}
A_{1}(1)&=&\int_{B_{g}(x_{0},\rho/\sqrt{K})}|{\rm Rm}|\phi^{2}dV_{t} \ \ \leq \ \ \int_{B_{g}(x_{0},\rho/\sqrt{K})}
\left(\frac{1}{3}|{\rm Rm}|^{3}\phi^{6}+\frac{1}{3/2}\right)dV_{t}\\
&\leq&\frac{1}{3}A_{1}(3)+{\rm Vol}_{g(t)}\left(B_{g}(x_{0},\rho/\sqrt{K})\right).
\end{eqnarray*}
Since $|\phi|\leq1$, it follows that
$$
A_{2}(1)\leq{\rm Vol}_{g(t)}\left(B_{g}(x_{0},\rho/\sqrt{K})\right).
$$
Using (\ref{2.50}) yields
$$
A_{4}(1)=\int_{B_{g}(x_{0},\rho/\sqrt{K})}|\nabla\phi|^{2}dV_{t}
\leq\frac{Ke^{CKT}}{\rho^{2}}{\rm Vol}_{g(t)}
\left(B_{g}(x_{0},\rho/\sqrt{K})\right).
$$
These estimates imply (\ref{2.51}).
\end{proof}

According to (\ref{2.37}) and (\ref{2.51}) we get
$$
\frac{d}{dt}\bigg[A_{1}(2m+1)
+\frac{C(1+L^{2})R(1+R)}{L^{2}}
D_{1}(2m+1)+C(1+L^{2})R(1+R)A_{2}(2m+1)
$$
$$
+ \ CRD_{2}(2m+1)+CR^{2}(1+R)(CR)^{m-1}D_{1}(1)
$$
$$
+ \ CR^{2}(1+R)\sum_{1\leq i\leq m-1}(CR)^{i}D_{1}(2i+1)
$$
$$
+ \ CR^{2}(1+R)\sum_{1\leq i\leq m-1}(CR)^{i}(1+L^{2})A_{2}(2i+1)\bigg]
$$
$$
\leq \ \ (CR)^{m+1}\bigg[A_{1}(2m+1)+A_{2}(2m+1)+A_{4}(2m+1)\bigg]
$$
$$
+ \ (CR)^{m+1}\left(1+\frac{K}{\rho^{2}}e^{CKT}\right)
{\rm Vol}_{g(t)}\left(B_{g}(x_{0},\rho/\sqrt{K})\right).
$$
Set
\begin{eqnarray}
U(m)&:=&A_{1}(2m+1)
+\frac{C(1+L^{2})R(1+R)}{L^{2}}
D_{1}(2m+1)\nonumber\\
&&+ \ C(1+L^{2})R(1+R)A_{2}(2m+1)\nonumber\\
&&+ \ CRD_{2}(2m+1)+CR^{2}(1+R)(CR)^{m-1}D_{1}(1)\label{2.52}\\
&&+ \ CR^{2}(1+R)\sum_{1\leq i\leq m-1}(CR)^{i}D_{1}(2i+1)\nonumber\\
&&+ \ CR^{2}(1+R)\sum_{1\leq i\leq m-1}(CR)^{i}(1+L^{2})A_{2}(2i+1).\nonumber
\end{eqnarray}
Then
\begin{eqnarray}
\frac{d}{dt}U(m)&\leq&(CR)^{m+1}
\left[U(m)+\frac{1}{C(1+L^{2}R(1+R)}U(m)
+A_{4}(2m+1)\right]\nonumber\\
&&+ \ (CR)^{m+1}\left(1+\frac{K}{\rho^{2}}e^{CKT}\right)
{\rm Vol}_{g(t)}\left(B_{g}(x_{0},\rho/\sqrt{K})\right).\label{2.53}
\end{eqnarray}
For $A_{4}(2m+1)$, one can derive
\begin{eqnarray}
A_{4}(2m+1)&=&\int_{B_{g}(x_{0},\rho/\sqrt{K})}
|{\rm Rm}|^{2m}|\nabla\phi|^{2}
\phi^{4m}dV_{t}\nonumber\\
&\leq&\frac{CK}{\rho^{2}}e^{CKT}
\int_{B_{g}(x_{0},\rho/\sqrt{K})}|{\rm Rm}|^{2m}\phi^{4m}dV_{t}\label{2.54}\\
&\leq&\frac{CK}{\rho^{2}}e^{CKT}
\int_{B_{g}(x_{0},\rho/\sqrt{K})}
\left[\frac{(|{\rm Rm}|^{2m}\phi^{4m})^{\frac{2m+1}{2m}}}{\frac{2m+1}{2m}}
+\frac{1}{2m+1}\right]dV_{t}\nonumber\\
&\leq&\frac{CK}{\rho^{2}}e^{CKT}
A_{1}(2m+1)+\frac{CK}{\rho^{2}}e^{CKT}
{\rm Vol}_{g(t)}\left(B_{g}(x_{0},\rho/\sqrt{K})\right).\nonumber
\end{eqnarray}
Plugging (\ref{2.54}) into (\ref{2.53}) implies
\begin{equation}
\frac{d}{dt}U(m)\leq(CR)^{m+1}
\left(1+\frac{K}{\rho^{2}}e^{CKT}\right)
\left[U(m)+{\rm Vol}_{g(t)}\left(B_{g}(x_{0},\rho/\sqrt{K})\right)\right].\label{2.55}
\end{equation}
As in \cite{KMW, L1, LY}, we can derive from (\ref{2.55}) the follolwing estimate
\begin{eqnarray}
U(m)&\leq& e^{\Xi T}\left[U(m)\bigg|_{t=0}+e^{CKT}{\rm Vol}_{g(t)}\left(B_{g}(x_{0},\rho/\sqrt{K})\right)\right]\nonumber\\
&\leq&e^{\Xi T}\left[U(m)\bigg|_{t=0}+e^{CKT}{\rm Vol}_{g}\left(B_{g}(x_{0},\rho/\sqrt{K})\right)\right]\label{2.56}
\end{eqnarray}
with
\begin{equation}
\Xi:=(CR)^{m+1}
\left(1+\frac{K}{\rho^{2}}e^{CKT}\right).\label{2.57}
\end{equation}
From the definition (\ref{2.32}) of $R$ and (\ref{2.52}) we find that
\begin{eqnarray*}
U(m)\bigg|_{t=0}&\leq&\int_{B_{g}(x_{0},\rho/\sqrt{K})}
|{\rm Rm}_{g}|^{2m+1}_{g}dV_{g}+CR^{2}\int_{B_{g}(x_{0},\rho/\sqrt{K})}|{\rm Rm}_{g}|^{2m}_{g}|F|^{2}_{g}dV_{g}\\
&&+ \ CR^{4}\int_{B_{g}(x_{0},\rho/\sqrt{K})}|{\rm Rm}_{g}|^{2m}_{g}dV_{g}
+(CR)^{m+2}\int_{B_{g}(x_{0},\rho/\sqrt{K})}|F|^{2}_{g}dV_{g}\\
&&+ \ CR\int_{B_{g}(x_{0},\rho/\sqrt{K})}|{\rm Ric}_{g}|^{2}_{g}|{\rm Rm}_{g}|^{2m}_{g}dV_{g}\\
&&+ \ (CR)^{m+2}
\sum_{1\leq i\leq m-1}\int_{B_{g}(x_{0},\rho/\sqrt{K})}|{\rm Rm}_{g}|^{2i}_{g}|F|^{2}_{g}dV_{g}\\
&&+ \ (1+L^{2})(CR)^{m+2}
\sum_{1\leq i\leq m-1}\int_{B_{g}(x_{0},\rho/\sqrt{K})}|{\rm Rm}_{g}|^{2i}_{g}dV_{g}\\
&\leq&\int_{B_{g}(x_{0},\rho/\sqrt{K})}|{\rm Rm}_{g}|^{2m+1}_{g}dV_{g}\\
&&+ \ (CR)^{m+3}\sum_{0\leq i\leq m}\int_{B_{g}(x_{0},\rho/\sqrt{K})}|{\rm Rm}_{g}|^{2i}_{g}dV_{g}.
\end{eqnarray*}
In particular
\begin{eqnarray}
\int_{B_{g}(x_{0},\rho/2\sqrt{K})}|{\rm Rm}_{g(t)}|^{2m+1}_{g(t)}dV_{g(t)}
&\leq&e^{\Xi T}\bigg[\int_{B_{g}(x_{0},\rho/\sqrt{K})}|{\rm Rm}_{g}|^{2m+1}_{g}dV_{g}\nonumber\\
&&+ \ (CR)^{m+3}\sum_{0\leq i\leq m}\int_{B_{g}(x_{0},\rho/\sqrt{K})}|{\rm Rm}_{g}|^{2i}_{g}dV_{g}\label{2.58}\\
&&+ \ e^{CRT}{\rm Vol}_{g}\left(B_{g}(x_{0},\rho/\sqrt{K})\right)\bigg].\nonumber
\end{eqnarray}
now with
$$
\Xi=(CR)^{m+1}
\left(1+\frac{R}{\rho^{2}}e^{CRT}\right).
$$
Setting
\begin{equation}
    \Lambda:=\max_{B_{g}(x_{0},\rho/\sqrt{K})}|{\rm Rm}_{g}|_{g}
\end{equation}
we obtain
$$
\int_{B_{g}(x_{0},\rho/2\sqrt{K})}|{\rm Rm}_{g(t)}|^{2m+1}_{g(t)}dV_{g(t)} \ \ 
\leq \ \ e^{\Xi T}{\rm Vol}_{g}\left(B_{g}(x_{0},\rho/\sqrt{K})\right)
$$
$$
\cdot\bigg[\Lambda^{2m+1}
+(CR)^{m+3}\sum_{0\leq i\leq m}\Lambda^{2i}
+e^{CRT}\bigg]
$$
$$
\leq \ \ e^{\Xi T}{\rm Vol}_{g}\left(B_{g}(x_{0},\rho/\sqrt{K})\right)
\left[(CR)^{m+3}
(\max\{\Lambda, 1\})^{2m+1}+e^{CRT}\right].
$$

For any nonnegative time-dependent function $H$, we set
$$
-\kern-11pt\int_{B_{g}(x_0,\rho/2\sqrt{K})}H\!\ dV_{g(t)}:=\frac{1}{{\rm Vol}_{g}(B_{g}(x_{0},\rho/2\sqrt{K}))}
    \int_{B_{g}(x_{0},\rho/2\sqrt{K})}H\!\ dV_{g(t)}
$$

\begin{lemma}\label{l2.8} One has
\begin{equation}
    \left(-\kern-11pt\int_{B_{g}(x_0,\rho/2\sqrt{K})}
    |{\rm Rm}_{g(t)}|^{2m+1}_{g(t)}dV_{g(t)}\right)^{\frac{1}{2m+1}}\leq C e^{\Xi(T+\rho)}\left[R^{2}\max\{\Lambda, 1\}+e^{CRT}\right].\label{2.60}
\end{equation}
\end{lemma}

\begin{proof} Observe that
$$
\frac{1}{{\rm Vol}_{g}(B_{g}(x_{0},\rho/2\sqrt{K}))}
    \int_{B_{g}(x_{0},\rho/2\sqrt{K})}
    |{\rm Rm}_{g(t)}|^{2m+1}_{g(t)}dV_{g(t)}
$$
$$
\leq e^{\Xi T}\frac{{\rm Vol}_{g}(B_{g}(x_{0},\rho/\sqrt{K}))}{{\rm Vol}_{g}(B_{g}(x_{0},\rho/2\sqrt{K}))}
\left[(CR)^{m+3}(\max\{\Lambda, 1\})^{2m+1}
+e^{CRT}\right]
$$
$$
\leq C e^{\Xi (T+\rho)}\left[(CR)^{m+3}(\max\{\Lambda, 1\})^{2m+1}
+e^{CRT}\right]
$$
by the Bishop-Gromov volume comparision theorem.
\end{proof}

%%%%%%%%%%%%%%%%%%%%%%%%%%%%%%%%%%%%%%%%%%%%%%%%%%%%%%%%%%%%%%%%%%%%%%%%%%%%%%
\section{Proof of Theorem \ref{t1.2}}\label{section3}
%%%%%%%%%%%%%%%%%%%%%%%%%%%%%%%%%%%%%%%%%%%%%%%%%%%%%%%%%%%%%%%%%%%%%%%%%%%%%%

Suppose now $T_{\max}<+\infty$ and
$$
|{\rm Ric}_{g(t)}|_{g(t)}\leq K, \ \ \ |F(t)|_{g(t)}\leq L, \ \ \ |\nabla_{g(t)}F(t)|_{g(t)}\leq P \ \ \ \text{on} \ M\times[0,T_{\max}),
$$
for some constants $K$, $L$ and $P$. We write
$$
f:=Cu, \ \ \ u:=1+|{\rm Rm}_{g(t)}|^{2}_{g(t)}+|F(t)|^{2}_{g(t)}
    +|\nabla_{g(t)}F(t)|^{2}_{g(t)}
$$
in (\ref{2.12}) so that
\begin{equation}
    -\Box u\geq -fu.\label{3.1}
\end{equation}
According to Lemma \ref{l2.8}, we have
\begin{equation}
    -\kern-11pt\int_{B_{g}(x_0,\rho/2\sqrt{K})}f^{m+\frac{1}{2}}dV_{g(t)}\leq C(m,K, L, P, \rho, T_{\max},\Lambda)<+\infty\label{3.2}
\end{equation}
for any integer $m\geq1$. If we apply Lemma 19.1 in \cite{PeterLi} to (\ref{3.1}) together with (\ref{3.2}), as well as the proof in \cite{KMW, L1, LY}, we obtain particularly
$$
\max_{B_{g}(x_{0},\rho/4\sqrt{K})\times[0,T_{\max})}
|{\rm Rm}_{g(t)}|_{g(t)}\leq C'(m, K, L, P, \rho, T_{\max}, \Lambda)<+\infty
$$
Since $M$ is closed, it follows that
$$
\max_{M\times[0,T_{\max})}
|{\rm Rm}_{g(t)}|_{g(t)}\leq C'(m, K, L, P, \rho, T_{\max}, \Lambda)<+\infty
$$
which contradicts with (\ref{1.3}). Hence the conclusion (\ref{1.4}) holds.

%\begin{remark}\label{r3.1} If we can prove the local curvature integral bounds for $|\nabla F|^{2}$, as shown in %\cite{KMW}, together with Proposition \ref{p2.1}, Lemma \ref{l2.8} can be used to prove Theorem \ref{t1.2} in the %complete noncompact setting.
%\end{remark}

%%%%%%%%%%%%%%%%%%%%%%%%%%%%%%%%%%%%%%%%%%%%%%%%%%%%%%%%%%%%%%%%%%%%%%%%%%%%%%

\end{document}